\documentclass[leqno, letterpaper, 11pt]{article}
    \usepackage{amsmath}
    \usepackage{amsfonts} 
    \usepackage{fullpage} 
    \usepackage{graphicx}
	\usepackage{amssymb}
	\usepackage{mathrsfs}
	\usepackage{hyperref}
	\usepackage{cleveref}
	\usepackage{float}
	\usepackage{epstopdf}
	\usepackage{tikz-cd}
	\usepackage{pgf, tikz}
	\usepackage{thmtools}
	\usepackage{bm}
	\usepackage[numbers]{natbib}
	\usepackage{multicol}
\usepackage{thm-restate}
\usetikzlibrary{arrows, automata}
\usepackage[T1]{fontenc}
\usepackage[font=small,labelfont=bf,tableposition=top]{caption}
 
\DeclareCaptionLabelFormat{andtable}{#1~#2  \&  \tablename~\thetable}

\usepackage{amssymb}
\usepackage{lineno}

\usepackage{graphicx}
\usepackage{subfig}
\usepackage{amssymb}
\usepackage{amsmath,amsfonts,amssymb,amsthm}
\usepackage{graphicx,epsfig,setspace}
\usepackage{mathrsfs}
\onehalfspace
\newcommand\Tau{\mathcal{T}}

\usepackage{etoolbox}
\patchcmd{\thebibliography}{\section*{\refname}}{}{}{}

\newcommand{\Z}{\mathbb{Z}}

\newcommand{\cO}{\mathcal{O}}

\newcommand{\ord}{\text{\rm ord}}
\newcommand{\lcm}{\text{\rm lcm}}

\newcommand{\Poi}{\text{\rm Poisson}}
\newcommand\blfootnote[1]{%
  \begingroup
  \renewcommand\thefootnote{}\footnote{#1}%
  \addtocounter{footnote}{-1}%
  \endgroup
}

\theoremstyle{plain}
\newtheorem{thm}{Theorem}

\newtheorem{lem}{Lemma}[section]
\newtheorem{cor}[thm]{Corollary}
\newtheorem{prop}[lem]{Proposition}
\newtheorem{alg}[lem]{Algorithm}

\theoremstyle{definition}
\newtheorem{defi}[lem]{Definition}

\newtheorem{ques}[lem]{Question}

 \newtheorem*{prob*}{Problem}

\theoremstyle{remark}

\begin{document}

\title{Periodic solutions of one-dimensional cellular automata with random rules}
\author{{\sc Janko Gravner} and {\sc Xiaochen Liu}\\
Department of Mathematics\\University of California\\Davis, CA 95616\\{\tt gravner{@}math.ucdavis.edu, xchliu{@}math.ucdavis.edu}}

\maketitle
\bibliographystyle{plain}

\begin{abstract}
We study cellular automata with randomly selected rules.
Our setting are two-neighbor rules with a large number $n$ of states.
The main quantity we analyze is the asymptotic probability, as $n \to \infty$, that the random rule has a periodic solution with given spatial and temporal periods.
We prove that this limiting probability is non-trivial when the spatial and temporal periods are confined to a finite range.
The main tool we use is the Chen-Stein method for Poisson approximation.
The limiting probability distribution of the smallest temporal period for a given spatial period is deduced as a corollary and relevant empirical simulations are presented.
\end{abstract}

\blfootnote{\emph{Keywords}: Cellular automaton, periodic 
solution, Poisson approximation, random rule.}
\blfootnote{AMS MSC 2010:  60K35, 37B15, 68Q80.}


\section{Introduction}
We investigate one-dimensional cellular automata (CA), a class of temporally and spatially discrete dynamical systems, in which the update rule is selected at random.
Our focus is the asymptotic behavior of the probability that such random CA has a periodic solution with fixed spatial and temporal periods, as $n$, the number of states, goes to infinity.
This complements the work in \cite{gl2}, where the limiting behavior of the longest temporal period with a given spatial period is explored.
We assume the simplest nontrivial setting of two-neighbor rules.

The (\textbf{spatial}) \textbf{configuration} at time $t$ 
of a one-dimensional CA with \textbf{number of states} $n$ 
is a function 
$\xi_t$ that assigns to every site $x\in \Z$ its
\textbf{state} $\xi_t(x) \in \mathbb{Z}_n = \{0, 1, \dots, n - 1\}$.
The evolution of spatial configurations is given by 
a local 2-neighbor \textbf{rule} $f: \mathbb{Z}_n^2 \to \mathbb{Z}_n$ that updates $\xi_t$ to $\xi_{t+1}$ as follows: 
$$\xi_{t+1}(x) = f(\xi_{t}(x - 1), \xi_{t}(x)), \quad \text{for all }x \in \mathbb{Z}.$$
We abbreviate $f(a, b) = c$ as $a\underline{b} \mapsto c$.
We give a rule by listing its values for all pairs in reverse alphabetical order, from $(n - 1, n - 1)$ to $(0, 0)$.

Given $\xi_0$, the update rule determines the \textbf{trajectory} $\xi_t$, $t\in\Z_+=\{0,1,\ldots\}$, or, equivalently, the
the \textbf{space-time configuration}, which is the map $(x, t)\mapsto\xi_t(x)$ from 
$\mathbb{Z} \times \mathbb{Z}_+$ to $\mathbb{Z}_n$. By convention, a picture of this map is a painted grid, in which the temporal axis is oriented downward, the spatial axis is oriented rightward, and each state is given as a different color. 
To give an example, a piece of the space-time configuration is presented in Figure \ref{figure: PS and tile}.
In this figure, we have three states, i.e., $n = 3$, and the rule is \textit{021102022}, i.e.,
$2\underline{2}\mapsto 0, 
2\underline{1}\mapsto 2, 
2\underline{0}\mapsto 1, 
1\underline{2}\mapsto 1, 
1\underline{1}\mapsto 0, 
1\underline{0}\mapsto 2, 
0\underline{2}\mapsto 0, 
0\underline{1}\mapsto 2$ and 
$0\underline{0}\mapsto 2$.

\begin{figure}[h] 
    \centering
    \includegraphics[scale = 0.1, trim = 5cm 1cm 3cm 1cm, clip]{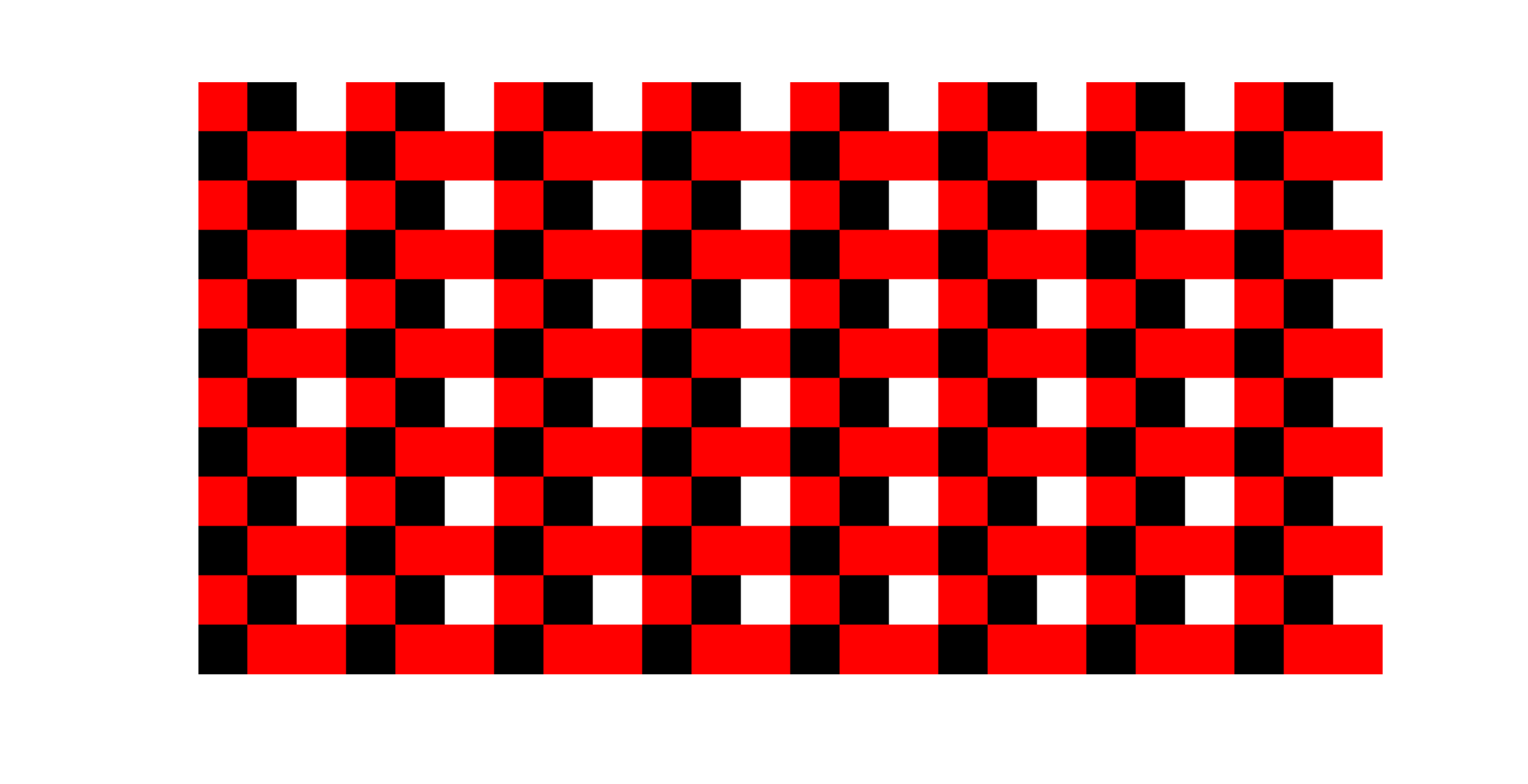}\\


    \caption{A piece of the space-time configuration of a 3-state rule. 
    In the space-time configuration, 0, 1 and 2 are represented by white, red and black cells, respectively.}
    \label{figure: PS and tile}
  \end{figure}
The space-time configuration in Figure \ref{figure: PS and tile} exhibits periodicity in both space and time.
In the literature \cite{boyle2007jointly}, such a configuration is called doubly or jointly periodic.
Since these are the only objects we study, we simply refer to such a configuration as a \textbf{periodic solution} (PS).
To be precise, start with a periodic spatial configuration $\xi_0$, such that there is a $\sigma > 0$ satisfying $\xi_0(x) = \xi_0(x + \sigma)$, for all $x \in \mathbb{Z}$.
Run a CA rule $f$ starting with $\xi_0$.
If we have $\xi_\tau(x) = \xi_0(x)$, for all $x \in \mathbb{Z}$ and that $\sigma$ and $\tau$ are both minimal, then we have found a periodic solution of \textbf{temporal period} $\tau$ and \textbf{spatial period} $\sigma$.
A \textbf{tile} is any rectangle with $\tau$ rows and $\sigma$ columns within this space-time configuration.
We interpret a tile as a configuration on a discrete torus;
we will not distinguish between spatial and temporal translations of a PS, and therefore between either rotations of a tile. The tile of a PS is by definition unique and we 
will identify a PS with its tile. 
As an example, in Figure \ref{figure: PS and tile}, we start with the initial configuration $\xi_0 = 120^\infty = \dots 120120120\dots$ (we give a configuration as a bi-infinite 
sequence when the position of the origin is clear or unimportant). 
After 2 updates, we have $\xi_2(x) = \xi_0(x)$, for all $x \in \mathbb{Z}$, thus the PS has temporal period 2 and spatial period 3.
Its tile is 
$
\begin{matrix}
1 & 2 & 0\\
2 & 1 & 1\\
\end{matrix}
$
.

CA that exhibit temporally periodic or jointly periodic behavior have been explored to some degree in the literature, and we give a brief review of some highlights. The foundational work 
is commonly considered to be~\cite{martin1984algebraic}. This 
paper, together with its successors~\cite{jen1, jen2}, focuses on algebraic methods to investigate additive CA, but also 
lays the foundation for more general rules. 
  More recent papers on temporal periodicity of additive binary rules include~\cite{cordovil1986periodic} 
  and~\cite{misiurewicz2006iterations}. The literature on 
non-additive rules is more scarce, but includes notable 
works \cite{boyle2007jointly} and \cite{boyle1999periodic} on 
the density of periodic configurations, which use both rigorous and experimental methods. 
A method of finding temporally periodic trajectories is discussed in \cite{xu2009dynamical}, which reiterates the 
utility of the relation between periodic configurations and cycles on graphs induced by the CA rules, introduced in \cite{martin1984algebraic}. This approach is useful
in the present paper as well.
Papers investigating long temporal periods of CA also include \cite{stevens1993transient, stevens1999construction}, as well as our companion papers \cite{gl2,gl3}. To mention another 
take on periodicity, the paper \cite{gravner2012robust} 
introduces \textit{robust} PS, which are those 
that expand into any environment with positive speed, 
and investigates their existence in all range 2 (i.e., 3-neighbor) binary CA.

We now present a formal setting to investigate PS from random rules, which, to our knowledge, have not been explored before.
Our rule space $\Omega_n$ consists of $n^{n^2}$ rules and we assign a uniform probability $\mathbb{P}$ to each rule $f$, therefore $\mathbb{P}(\{f\}) = 1/n^{n^2}$.
Let $\mathcal{P}_{\tau, \sigma, n}$ be the random set of PS with temporal period $\tau$ and spatial period $\sigma$ of such a randomly chosen CA rule.
The main quantity we are interested in is $\lim\mathbb{P}\left( \mathcal{P}_{\tau, \sigma, n} \neq \emptyset \right)$ as $n \to \infty$ for a fixed pair of $(\tau, \sigma)$.
In words, our focus is the limiting probability that a random CA rule has a PS with given temporal and spatial periods.
In the following theorem, we prove that this limit is nontrivial for any $\tau$ and $\sigma$. Define 
\begin{equation}\label{equation: lambda}
\displaystyle \lambda_{\tau, \sigma} = 
\frac{1}{\tau\sigma}\sum_{d \bigm| \gcd(\tau, \sigma)} \varphi(d)d,
\end{equation}
where $\varphi$ the Euler totient function.

\begin{thm}\label{theorem: main 1}
For any fixed integers $\tau\ge 1$ and $\sigma \ge 1$, $\mathbb{P}\left(\mathcal{P}_{\tau, \sigma, n} \neq \emptyset \right) \to 1 - \exp\left(-\lambda_{\tau, \sigma} \right)$ as $n \to \infty$.
\end{thm}

We also prove a more general result concerns the number of PS with a range of periods.
Assume $\Tau, \Sigma\subset \mathbb{N} = \{1, 2, \dots\}$,  and define 
$\mathcal{P}_{\Tau, \Sigma, n}=\mathcal{P}_{\Tau, \Sigma, n} (f) = \bigcup_{(\tau, \sigma) \in \Tau \times \Sigma} \mathcal{P}_{\tau, \sigma, n}$ and 
\begin{equation}\label{equation: Lambda}
\displaystyle \lambda_{\Tau, \Sigma} = \sum_{(\tau, \sigma) \in \Tau\times \Sigma}
\lambda_{\tau, \sigma}.
\end{equation}

\begin{thm}\label{theorem: main 2}
For a finite $\Tau\times \Sigma \subset \mathbb{N} \times \mathbb{N}$, 
$\mathbb{P}\left(\mathcal{P}_{\Tau, \Sigma, n} \neq \emptyset \right) \to 1 - \exp\left( -\lambda_{\Tau, \Sigma} \right)$ as $n \to \infty$.
\end{thm}

We define the random variable
$$Y_{\sigma, n} = \min\{\tau: \mathcal{P}_{\tau, \sigma, n} \neq \emptyset\}$$
to be the smallest temporal period of a PS with spatial period $\sigma$ of a randomly selected $n$-state rule. 
Figure~\ref{figure: 4 examples} provides four examples of rules $f$, with $Y_{4, 3}(f) = 1, 2, 3$ and $4$.
As a consequence of Theorem \ref{theorem: main 2}, for a given $\sigma > 0$, the random variable $Y_{\sigma, n}$ is stochastically bounded, in the sense of the following corollary.

\begin{figure}%
    \centering
    \subfloat[\textit{012200210}]{{\includegraphics[width=5cm, clip]{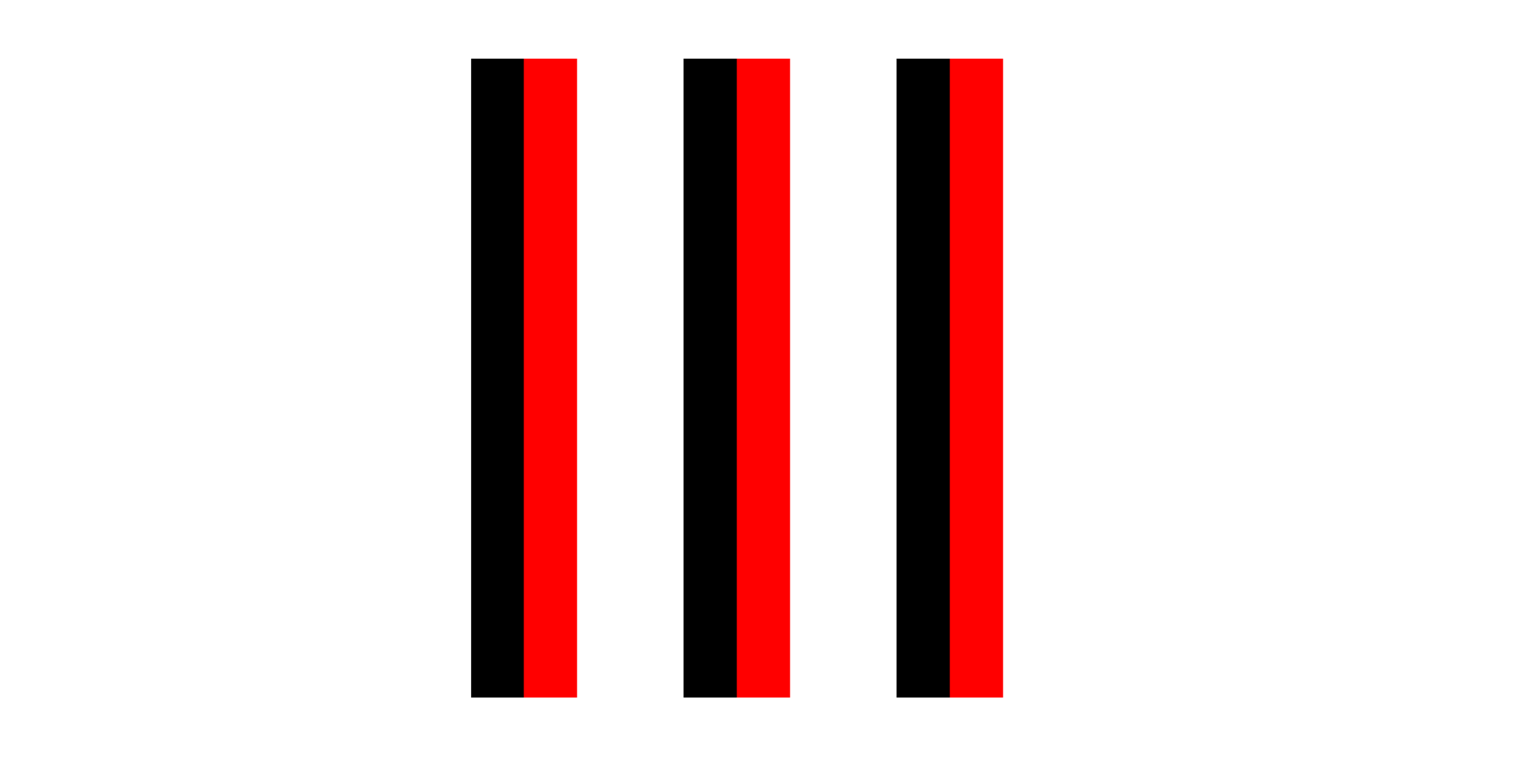} }}%
    \hspace{-1.5cm}
    \subfloat[\textit{021102120}]{{\includegraphics[width=5cm, clip]{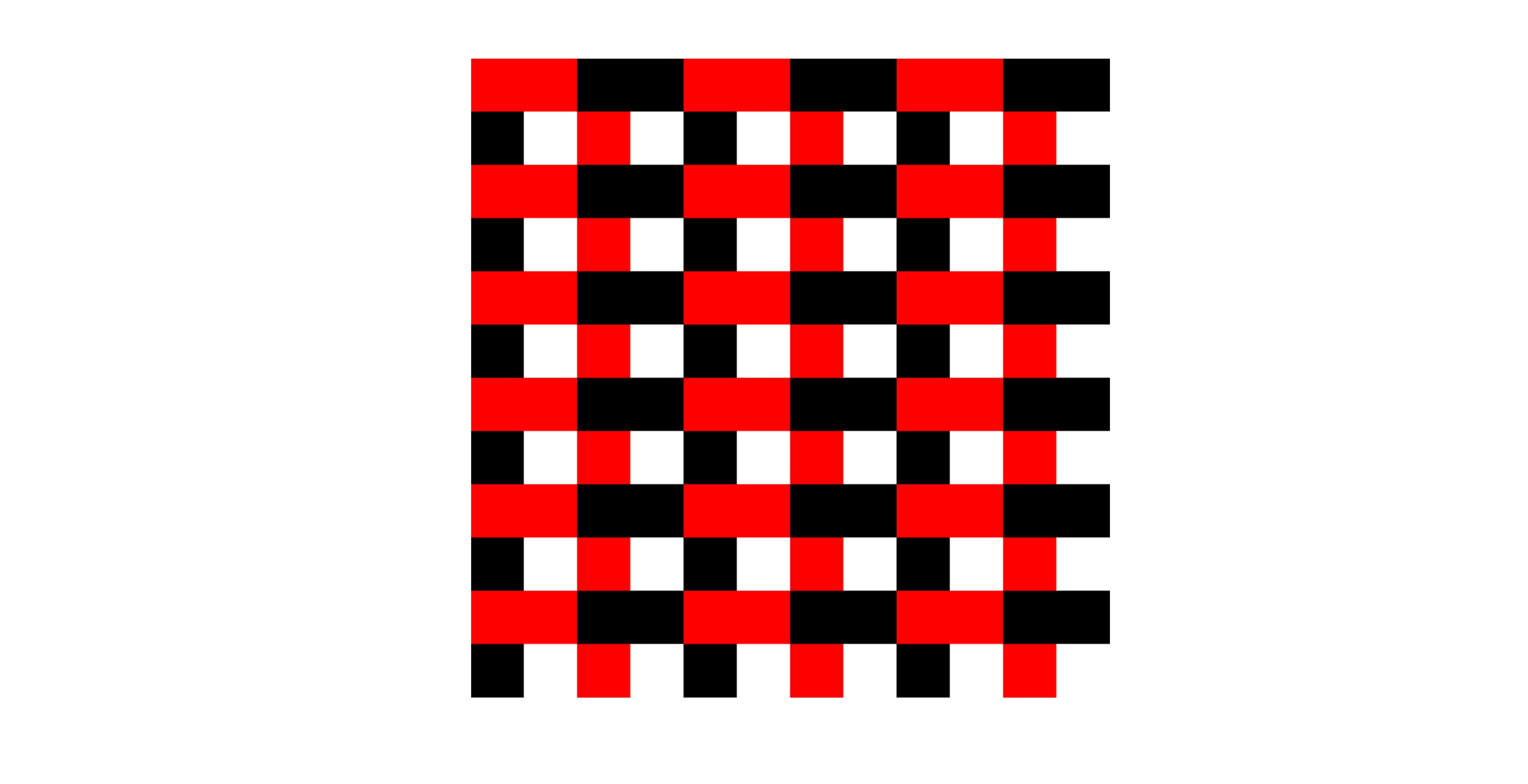} }}%
    \hspace{-1.5cm}
    \subfloat[\textit{100112122}]{{\includegraphics[width=5cm, clip]{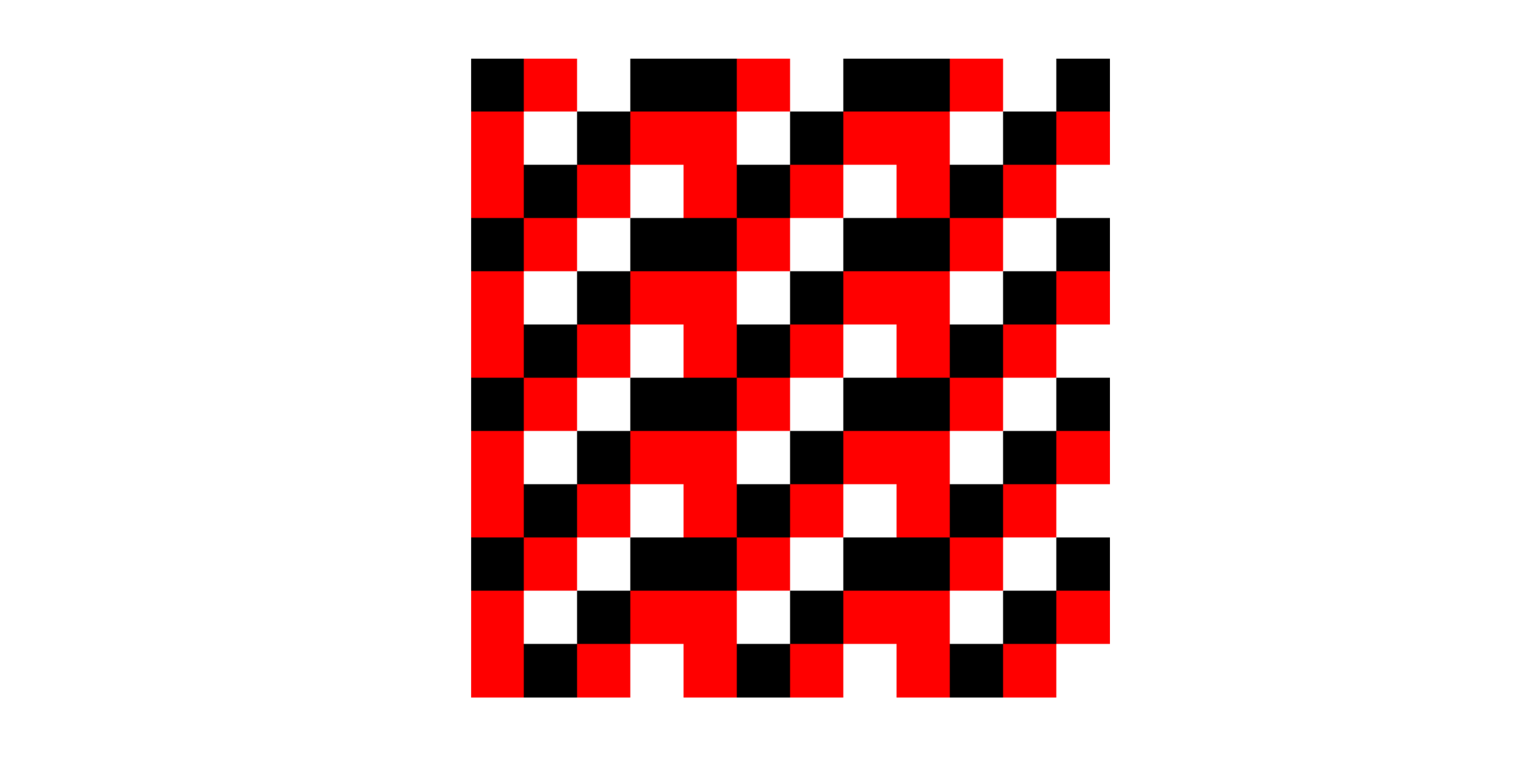} }}
    \hspace{-1.5cm}
    \subfloat[\textit{101201021}]{{\includegraphics[width=5cm, clip]{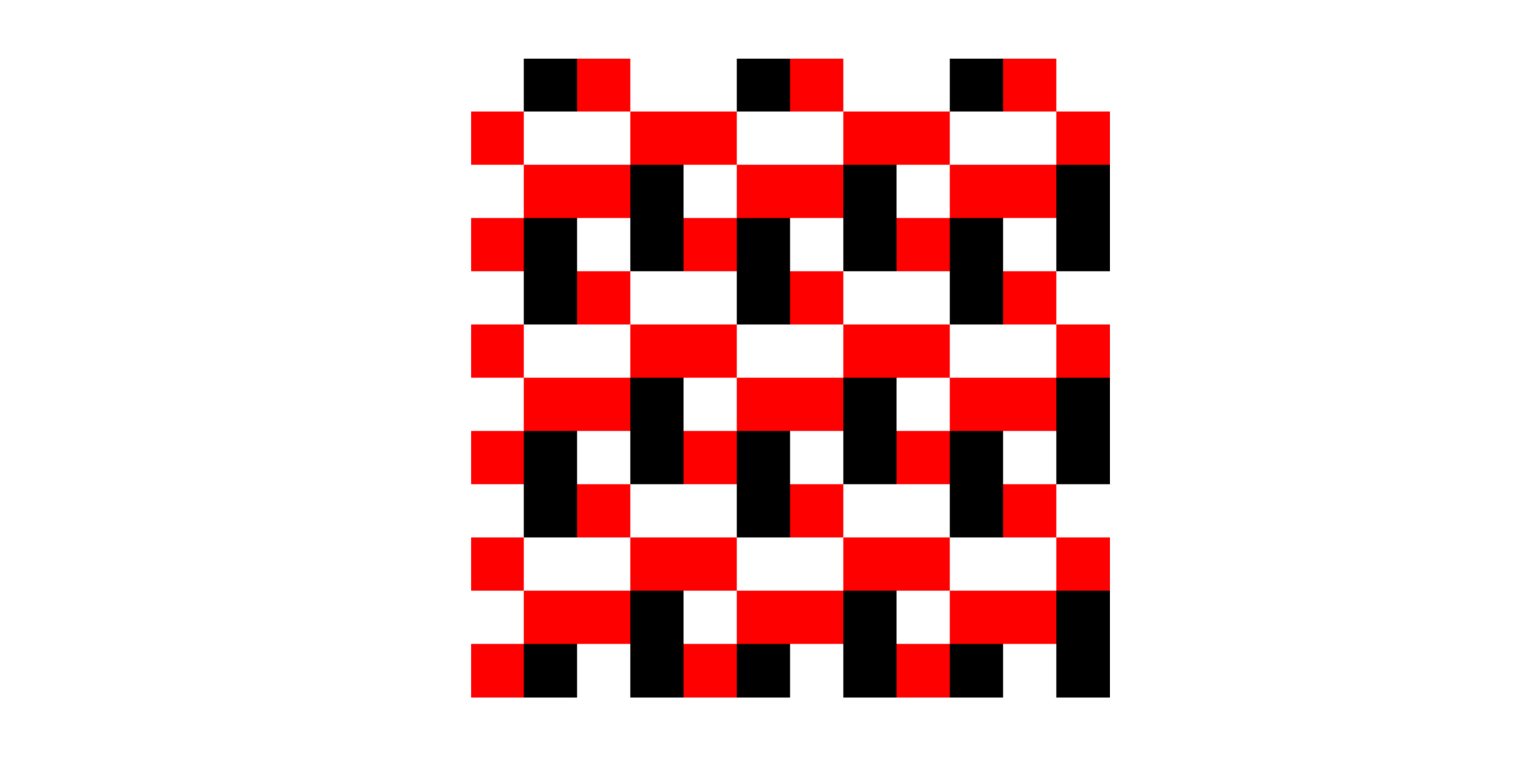} }}
    
    \caption{Pieces of PS for $\sigma = 4$ and 3-state rules, \textit{012200210}, \textit{021102120}, \textit{100112122} and \textit{101201021}, with temporal period $\tau = 1, 2, 3$ and $4$, respectively. (See the discussion before Corollary~\ref{cor:Y-prime}.)
    These temporal periods are the smallest in each case, as verified by Algorithm \ref{algorithm: configuration digraph} in Section \ref{section: configuration directed graph}. 
    Algorithm \ref{algorithm: PS} in Section \ref{section: label directed graph} shows that $\sigma = 4$ is not the minimal spatial period of PS given the corresponding temporal period $\tau = 1, 2$ and $3$ in the first three rules, while for the last rule $\sigma = 4$ is also the minimal spatial period of PS for temporal period $\tau = 4$.}%
    \label{figure: 4 examples}%
\end{figure}

\begin{cor} \label{corollary: min period}
The random variable $Y_{\sigma, n}$ converges weakly to a nontrivial distribution as $n \to \infty$.
\end{cor}

We now briefly discuss the relation between this corollary and the main results of \cite{gl2} and \cite{gl3}.
In \cite{gl2}, we consider a more general setting of CA rules with $r$ neighbors, that is, $\xi_t$ updates to $\xi_{t + 1}$ according to the rule $f: \mathbb{Z}_n^r \to \mathbb{Z}_n$, so that
$$\xi_{t + 1}(x) = f(\xi_{t}(x - r + 1), \dots, \xi_{t}(x)), \quad \text{for all }x \in \mathbb{Z}.$$
Fix a spatial period $\sigma$ and an $r$.
Let $X_{\sigma, n} = \max\{\tau: \mathcal{P}_{\tau, \sigma, n} \neq \emptyset\}$ be the largest temporal period of a PS with spatial period $\sigma$ of a randomly selected $r$-neighbor rule.
In the case when $\sigma \le r$, we prove that $X_{\sigma, n}/n^{\sigma/2}$ converges in distribution to a nontrivial limit, as $n \to \infty$. We also provide empirical evidence that 
the same result holds when $\sigma> r$, although in that 
case we do not have a rigorous proof even for $r=2$. At least 
for $r = \sigma = 2$, therefore, the shortest temporal period is stochastically bounded while the longest is on the order of $n$. 
Moreover, it is not hard to see that the maximum of the random variable $Y_{2, n}$ and $X_{2, n}$ are both $n^2 - n$. 
More generally, in \cite{gl3} and \cite{xiaochen}, 
we construct rules $f$ with $Y_{\sigma, n}(f) \ge C(\sigma)n^\sigma$. That is, the maximum of the random variable $Y_{\sigma, n}$ is of the same order as its upper bound $n^\sigma - \cO(n^{\sigma/2})$,  guaranteed by the pigeonhole principle.


In the next section, we collect 
our main tools: tiles of PS; 
circular shifts; 
oriented graphs induced by a rule; and the Chen-Stein method.
In Section 3, we discuss a class of tiles that plays a central role.
We prove the main results in Section 4 and conclude with a   discussion and several unsolved problems in Section 5.

\section{Preliminaries} 
\subsection{Tiles of a PS}



We recall that the spatial and temporal periods $\sigma$ and $\tau$ are assumed to be minimal, so a tile cannot be divided into smaller identical pieces.  We now take a closer look into properties of tiles.

If we choose an element in a tile $T$ to be placed at the position $(0, 0)$, $T$ may be expressed as a matrix $T = (a_{i,j})_{i = 0,\dots, \tau-1, j = 0, \dots ,\sigma - 1}$.
We always interpret the two subscripts 
modulo $\tau$ and $\sigma$. The matrix is determined up to
a space-time rotation, but note that two different rotations
cannot produce the same matrix due to the minimality of $\sigma$ 
and $\tau$.  
We say that $a_{i,j}$ is an element in $T$, and 
write $a_{i,j} \in T$, when we want to refer to the 
element of the matrix at the position $(i,j)$, and use the notation $\mathtt{row}_i$ and $\mathtt{col}_j$ to denote the $i$th row and $j$th column of a tile $T$, again after we fix $a_{0, 0}$. All the properties we now introduce are independent
of the chosen rotation (as they must be, to be 
 meaningful).

Let $T_1$ and $T_2$ be two tiles and $a_{i, j}$, $b_{k, m}$ be elements in $T_1$ and $T_2$, respectively. 
We say that $T_1$ and $T_2$ are \textbf{orthogonal}, and 
denote this property by $T_1 \perp T_2$,  if
 $\left(a_{i,j}, a_{i, j+1}\right) \neq \left(b_{k, m}, b_{k, m + 1}\right)$ for $i, j, k, m \in \mathbb{Z}_+$. It is important
 to observe that in this case 
 the two assignments   
 $a_{i,j}\underline{a_{i,j+1}}\mapsto a_{i+1,j+1}$ is and  
 $b_{k,m}\underline{b_{k,m+1}}\mapsto b_{k+1, m+1}$
 occur independently.

We say that $T_1$ and $T_2$ are \textbf{disjoint}, and 
denote this property by $T_1 \cap T_2=\emptyset$, if $a_{i, j} \neq b_{k, m}$, for $i, j, k, m \in \mathbb{Z}_+$. 
Clearly, every pair of disjoint tiles is orthogonal, but not vice versa.

Let $s(T) =  \#\{a_{i,j}: a_{i,j} \in T\}$ be the number of different states in the tile. Furthermore, let $p(T) = \#\{(a_{i,j}, a_{i,j+1}): a_{i,j}, a_{i,j+1} \in T\}$
be the \textbf{assignment number} of $T$; this is the number of values of the rule $f$ specified by $T$.
Clearly, $p(T) \ge s(T)$, so we define $\ell=\ell(T)= p(T) - s(T)$ to be the $\textbf{lag}$ of $T$. 
We record a few immediate properties of a tile in the following Lemma.

\begin{lem} \label{lemma: properties of tile}
Let $T= (a_{i,j})_{i = 0,\dots, \tau-1, j = 0, \dots ,\sigma - 1}$ be the tile of a PS with periods $\tau$ and $\sigma$. Then $T$ satisfies the following properties: 
\begin{enumerate}
\item Uniqueness of assignment: if $(a_{i, j}, a_{i, j + 1}) = (a_{k, m}, a_{k, m + 1})$, then $a_{i+1, j+1} = a_{k+1, m+1}$. 
\item Aperiodicity of rows: each row of $T$ cannot be divided into smaller identical pieces.
\end{enumerate}

\end{lem}
\begin{proof}
Part 1 is clear since $T$ is generated by a CA rule. Part 2 follows from part 1 and the assumption that the spatial 
period of $T$ is minimal.
\end{proof}

By contrast, we remark that there {\it may\/} exist periodic columns in a tile of a PS. For example, note that 
the first column in Figure~\ref{figure: 4 examples}(d)
has period $2$ rather than $4=\tau$. 

\subsection{Circular shifts} \label{section: circular shifts}

In this section, we introduce circular shifts, operation on $Z_n^\sigma$ (or $Z_n^\tau$), the set of words of length
$\sigma$ (or $\tau$) from the alphabet $\mathbb{Z}_n$.
They will be useful in Section \ref{section: simple tiles}.

\begin{defi}\label{defination: circular shift}
Let $\mathbb{Z}_n^\sigma$ consist of all length-$\sigma$ words. 
A \textbf{circular shift} is a map $\pi: \mathbb{Z}_n^\sigma \to \mathbb{Z}_n^\sigma$, given by an $i \in \mathbb{Z}_+$ as follows: $\pi(a_0a_1\dots a_{\tau - 1}) =  a_i a_{i+1} \dots a_{i+\sigma-1}$, where the subscripts are modulo $\sigma$. 
The \textbf{order} of a circular shift $\pi$ is the smallest $k$ such that $\pi^k(A) = A$ for all $A \in \mathbb{Z}_n^\sigma$, and is denoted by $\ord(\pi)$.
Circular shifts on $\mathbb{Z}_n^{\tau}$ will also appear in the sequel and are defined in the same way.
\end{defi}


\begin{lem} \label{lemma: euler totient}
Let $\pi$ be a circular shift on $\mathbb{Z}_n^\sigma$ and let $A\in \mathbb{Z}_n^\sigma$ be an aperiodic length-$\sigma$ word from alphabet $\mathbb{Z}_n$. 
Then:
(1) $\ord(\pi)\bigm|\sigma$; and
(2) for any $d\bigm|\sigma$,
$$|\left\{B \in \mathbb{Z}_n^\sigma: A = \pi(B) \text{ for some }\pi \text{ with }\ord(\pi) = d \right\}|= \varphi(d).$$ 
\end{lem}
\begin{proof}
Note that the $\sigma$ circular shifts form a cyclic group of order $\sigma$.
Moreover, $\ord(\pi)$ of a circular shift is its order in the group, thus (1) follows. To prove (2), observe that
the circular shifts of order $d$ 
generate a cyclic subgroup and the number 
of them is $\varphi(d)$.
As $A$ is aperiodic, the cardinality in the claim is the same.
\end{proof}

We say that two words  $A$ and $B$ of length $\sigma$ are \textbf{equal up to a circular shift} if $B = \pi(A)$ for some circular shift $\pi$. For example, 
words  $0123$ and $2301$ are not equal, but are equal up to a circular shift.

\subsection{Directed graph on configurations} \label{section: configuration directed graph}

Connections between directed graphs on periodic 
configurations and cycles are well-established 
\cite{martin1984algebraic, wolfram2002new,
kim2009state, xu2009dynamical}, as they are 
useful for analysis of PS with a fixed spatial period.


\begin{defi}
Let $A = a_0\dots a_{\sigma-1}$ and $B=b_0\dots b_{\sigma-1}$ be two words from alphabet $\mathbb{Z}_n$.
We say that $A$ \textbf{down-extends to} $B$, if $f(a_{i}, a_{i + 1}) = b_{i + 1}$, for all $i \ge 0$, where (as usual) the indices are modulo $\sigma$.
\end{defi}

If $A$ down-extends to $B$, then $\pi(A)$ also down-extends to $\pi(B)$, for any circular shift  $\pi$ on $\mathbb{Z}_n^\sigma$. 
Therefore, we can define, for a fixed $\sigma$, the \textbf{configuration digraph} on 
equivalence classes of words equal up to circular shifts, which 
has an arc from $A$ to $B$ if $A$ down-extends to $B$ (where we
identify the equivalence class with any of its representatives). 
See Figure \ref{figure: configuration digraph} for the configuration digraph of the 3-state rule \textit{021102022}.
For instance, there is an arc from 122 to 210 as 
$1\underline{2} \mapsto 1$, 
$2\underline{2} \mapsto 0$ and 
$2\underline{1} \mapsto 2$. 
The following algorithm and self-evident proposition determine 
the PS  in Figure~\ref{figure: PS and tile} from
the length-2 cycle $120\leftrightarrow 211$  
in Figure~\ref{figure: configuration digraph}. 

\begin{alg}\label{algorithm: configuration digraph}
Input: Configuration digraph $D_{\sigma, f}$ of $f$ and spatial  period $\sigma$.\\
Step 1: Find all the directed cycles in $D_{\sigma, f}$.\\
Step 2: For each cycle $A_0 \to A_1 \to \cdots \to A_{\tau - 1}
\to A_0$, form the tile $T$ by 
placing configurations $A_0, A_1,\ldots,  A_{\tau-1}$ on successive rows. \\
Step 3: If the spatial period of $T$ is minimal, output $T$.
\end{alg}

\begin{prop}\label{proposition: configuration digraph gives all PS }
All PS of spatial period $\sigma$ of $f$ are obtained by Algorithm~\ref{algorithm: configuration digraph}.
\end{prop}

We remark that Step 3 in 
Algorithm~\ref{algorithm: configuration digraph} is necessary, as, for instance, the cycle $000\leftrightarrow 222$ in 
Figure~\ref{figure: configuration digraph} results in 
a PS of spatial period 1 instead of 3. In the same vein, 
the periods of configurations are non-increasing, and may decrease, along any directed path on the configuration digraph. For example, in Figure~\ref{figure: configuration digraph}, 
the configuration 100 down-extends to 222, thus the period is reduced from 3 to 1 and then remains 1. These period reductions
play a crucial role in our companion paper \cite{gl2}.  

\begin{figure}[!h]
\centering
\begin{tikzpicture}
[
            > = stealth, 
            shorten > = 1pt, 
            auto,
            node distance = 3cm, 
            semithick 
        ]

        \tikzstyle{every state}=[
            draw = black,
            thick,
            fill = white,
            minimum size = 4mm
        ]
        
        \node (1) at (0, 0) {$122$};
        \node (2) at (2, 0) {$210$};
        \node (3) at (4, 0) {$022$};
        \node (11) at (6, 0) {$101$};
        \node (4) at (6, -1) {$100$};
        \node (5) at (4, -1) {$222$};
        \node (6) at (2, -1) {$000$};
        \node (7) at (0, -1) {$111$};
        
        \node (8) at (0, -2) {$002$};
        \node (9) at (2, -2) {$120$};
        \node (10) at (4, -2) {$211$};
        
        \path[->] (1) edge node {} (2);
        \path[->] (2) edge node {} (3);
        \path[->] (3) edge node {} (4);
        \path[->] (4) edge node {} (5);
        \path[->] (5) edge [bend right=10]  node {} (6);
        \path[->] (6) edge [bend right=10]  node {} (5);
        \path[->] (7) edge node {} (6);
        \path[->] (11) edge node {} (3);
        
        \path[->] (8) edge node {} (9);
        \path[->] (9) edge [bend right=10] node {} (10);
        \path[->] (10) edge [bend right=10] node {} (9);
        
    \end{tikzpicture}
\caption{Configuration digraph  of the 3-state rule \textit{021102022} and spatial period $\sigma=3$.  }
\label{figure: configuration digraph}
\end{figure}
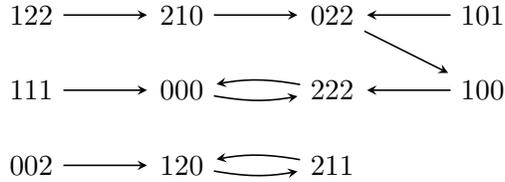

\subsection{Directed graph on labels} \label{section: label directed graph}

In this subsection, we fix the temporal period $\tau$, instead of the spatial period $\sigma$, and obtain another digraph induced by the rule. The construction below is an adaption of label trees from \cite{gravner2012robust}.
We call such a graph \textbf{label digraph}.

\begin{defi}
Let $A = a_0\dots a_{\tau-1}$ and $B=b_0\dots b_{\tau-1}$ be two words from alphabet $\mathbb{Z}_n$, which we 
call \textbf{labels} of length $\tau$.
(While it is best to view them as vertical columns, we write them horizontally for reasons of space, as in \cite{gravner2012robust}.)
We say that $A$ \textbf{right-extends to} $B$ if $f(a_{i}, b_{i}) = b_{i+1}$, for all $i \in \mathbb{Z}_+$, where (as usual) the indices are modulo $\tau$.
We form the \textbf{label digraph} associated with a given $\tau$ by forming an arc from a label $A$ to a label $B$ if $A$ right-extends to $B$.
\end{defi}

A label $A = a_0\dots a_{\tau - 1}$ right-extends to $B$ if and only if we preserve the temporal periodicity from a column $A$ to the column $B$ to its right.
This fact is the basis for the Algorithm \ref{algorithm: PS} below, which gives all the PS with temporal period $\tau$.
The label digraph of same rule as in Figure~\ref{figure: configuration digraph} and temporal period $\tau=2$ is presented in  Figure \ref{figure: label digraph}.
For example, we have the arc from label 12 to 10 as 
$1\underline{1} \mapsto 0$, $2\underline{0} \mapsto 1$. 
Either of the two 3-cycles in the digraph generates the PS in Figure \ref{figure: PS and tile}. 

\begin{alg}\label{algorithm: PS}
Input: Label digraph $D_{\tau, f}$ of $f$ with period $\tau$.\\
Step 1: Find all the directed cycles in $D_{\tau, f}$.\\
Step 2: For each cycle $A_0 \to A_1 \to \cdots \to A_{\sigma - 1}\to A_0$, form the tile $T$   
by placing configurations $A_0, A_1,\ldots,  A_{\sigma-1}$ on successive columns.\\
Step 3: If both spatial and temporal periods of $T$  are minimal, then output $T$.
\end{alg}

\begin{prop}
All PS of temporal period $\tau$ of $f$ can be obtained by the Algorithm \ref{algorithm: PS}.
\end{prop}

Again, Step 3 is necessary due to the same reason as Section \ref{section: configuration directed graph}. However, note 
the differences between the two graphs: the out-degrees
in Figure~\ref{figure: label digraph} are between $0$ and $3$, 
and the temporal periods are not necessarily non-decreasing along a directed path. For example, 00 right-extends to 02.
We also note that lifting the label
digraph to one on equivalence classes, although possible, makes 
cycles more obscure and is thus less convenient.  

\begin{figure}[!h]
\centering
\begin{tikzpicture}
[
            > = stealth, 
            shorten > = 1pt, 
            auto,
            node distance = 3cm, 
            semithick 
        ]

        \tikzstyle{every state}=[
            draw = black,
            thick,
            fill = white,
            minimum size = 4mm
        ]
        
        \node (1) at (0, 0) {$02$};
        \node (2) at (2, 0) {$12$};
        \node (3) at (4, 0) {$21$};
        \node (4) at (6, 0) {$20$};
        \node (5) at (3, 1) {$10$};
        \node (6) at (3, -1) {$01$};
        \node (7) at (3, -2) {$00$};
        
        \node (8) at (4, -2) {$11$};
        \node (9) at (5, -2) {$22$};

		\path[->] (1) edge [loop above] node {} (1);
		\path[->] (4) edge [loop above] node {} (4);
        \path[->] (2) edge node {} (1);
        \path[->] (2) edge[bend right=10] node {} (3);
        \path[->] (2) edge node {} (5);
        
        \path[->] (3) edge[bend right=10] node {} (2);
        \path[->] (3) edge node {} (4);
        \path[->] (3) edge node {} (6);
        
        \path[->] (4) edge [loop above] node {} (4);
        
        \path[->] (5) edge node {} (1);
        \path[->] (5) edge node {} (3);
        
        \path[->] (6) edge node {} (2);
        \path[->] (6) edge node {} (4);
        
        \path[->] (7) edge node {} (1);
        \path[->] (7) edge node {} (4);

    \end{tikzpicture}
\caption{Label digraph of the 3-state rule \textit{021102022} and temporal period $\tau=2$.}
\label{figure: label digraph}
\end{figure}
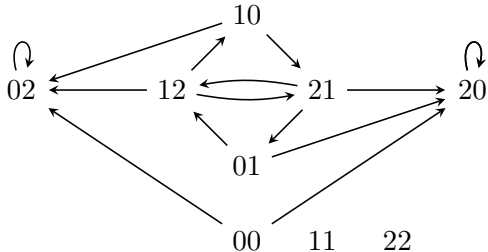

\subsection{Chen-Stein method for Poisson approximation}
The main tool we use to prove Poisson convergence is the Chen-Stein method \cite{barbour1992poisson}. 
We denote by $\Poi(\lambda)$ a Poisson random variable with expectation $\lambda$, and by $d_{\text{TV}}$ the total variation distance. 
We need the following setting for our purposes. Let 
$I_i$, for $i\in \Gamma$ be indicators of a finite family of events, which is indexed by $\Gamma$,
$p_i = \mathbb{E}(I_i)$, $W = \sum_{i\in \Gamma}I_i$, $\lambda = \sum_{i\in \Gamma} p_i = \mathbb{E}W$, and
$\Gamma_i = \{j\in \Gamma: j \neq i,\, I_i\text{ and } I_j \text{ are not independent}\}$. We quote Theorem 4.7 from \cite{ross2011fundamentals}.
\begin{lem} \label{lemma: chen-stein lemma}
We have
$$d_{\text{TV}}(W, \Poi(\lambda)) \le \min(1, \lambda^{-1}) \left[ \sum_{i\in \Gamma} p_i^2 + \sum_{i\in \Gamma, j\in \Gamma_i} \left(p_ip_j + \mathbb{E}\left(I_iI_j\right)\right)\right].$$
\end{lem}

In our applications of the above lemma, all deterministic and random quantities depend on the number $n$ of states, which we make explicit by the subscripts.
In our setting, we prove that $d_{\text{TV}}(W_n, \Poi(\lambda_n)) = \cO(1/n)$ and that $\lambda_n \to \lambda$ as $n \to \infty$, for an explicitly given $\lambda$, 
which implies that $W_n$ converges to $\Poi(\lambda)$ in distribution. See Theorem \ref{theorem: main 1} and Theorem \ref{theorem: main 2}.


\section{Simple tiles} \label{section: simple tiles}
We call a tile $T$ \textbf{simple} if its lag vanishes: 
$\ell(T) = p(T) - s(T) = 0$. 
It turns out that in $\mathbb{P}(\mathcal{P}_{\tau, \sigma, n} \neq \emptyset)$, the probability of existence of PS with simple tiles provides the dominant terms, thus this class of tiles is of 
central importance. For example, consider the tiles
$$T_1 = 
\begin{matrix}
0 & 1 & 2 & 3\\
2 & 3 & 0 & 1\\
\end{matrix},\qquad
T_2 = 
\begin{matrix}
0 & 1 & 2 & 1\\
2 & 1 & 0 & 1\\
\end{matrix}.
$$
Then $T_1$ is simple, as $s(T_1) = p(T_1) = 4$, 
but $T_2$ is not, as $s(T_2) = 3$ and $p(T_2) = 4$. Naturally, 
we call a PS simple if its tile is simple.

We denote by $\mathcal{P}_{\tau, \sigma, n}^{(\ell)}$ as the set of PS whose tile $T$ has lag $\ell$.
Thus the set of simple PS is $\mathcal{P}_{\tau, \sigma, n}^{(0)}$.
The following lemma addresses ramifications of repeated 
states in simple tiles.

\begin{lem}\label{lemma:simple-tiles-repeat}   
Assume $T = (a_{i,j})_{ i = 0,\dots, \tau-1, j = 0, \dots ,\sigma - 1}$ is a simple tile. Then
\begin{enumerate}
\item the states on each row of $T$ are distinct;
\item if two rows of $T$ share a state, then they are circular shifts of each other;
\item the states on each column of $T$ are distinct; and
\item if two columns of $T$ share a state, then they are circular shifts of each other.
\end{enumerate}
\end{lem}

\begin{proof}
\noindent{\it Part 1\/}: When $\sigma = 1$, each row contains only one state, making the claim trivial.
Now, assume that $\sigma \ge 2$ and that $a_{i, j} = a_{i, k}$ for some $i$ and $j \neq k$.
We must have $a_{i, j + 1} = a_{i, k + 1}$ in order to avoid $p(T) > s(T)$. 
Repeating this procedure for the remaining states on $\mathtt{row}_i$ shows that this row is periodic, contradicting part 2 of Lemma \ref{lemma: properties of tile}. 
 
\noindent{\it Part 2\/}: If $a_{i,j} = a_{k, m}$, for $i \neq k$, then the states to their right must agree, i.e., $a_{i,j+1} = a_{k, m+1}$, in order to avoid $p(T) > s(T)$. 
Repeating this observation for the remaining states on $\mathtt{row}_i$ and $\mathtt{row}_k$ gives the desired result. 
  
  \noindent{\it Part 3\/}:
Assume a column contains repeated state, say $a_{i, j} = a_{k, j}$ for some $i, j$ and $k$. 
By part 2, $\mathtt{row}_i$ is exactly the same as $\mathtt{row}_k$, so that the temporal period of this tile can be reduced, a contradiction.
 
\noindent{\it Part 4\/}:
Assume that $a_{i, j} = a_{k, m}$, for $j \neq m$. 
Then $a_{i, j + 1} = a_{k, m + 1}$ by parts 1 and 2. 
So, $a_{i + 1, j + 1} = a_{k + 1, m + 1}$ by part 1 in Lemma \ref{lemma: properties of tile}.
So, $a_{i + 1, j} = a_{k + 1, m}$, again by  parts 1 and 2.
Now, repeating the previous step for $a_{i + 1, j} = a_{k + 1, m}$ gives the desired result.
\end{proof}

We revisit the remark following Lemma \ref{lemma: properties of tile}:  a tile may have periodic columns, but such a 
tile cannot be simple.

Suppose a tile  $T = (a_{i,j})_{ i = 0,\dots, \tau-1, j = 0, \dots ,\sigma - 1}$ is simple. We will take a closer 
look with circular shifts of rows, so we fix a row, say
the first row $\mathtt{row}_0$.
(We could start with any row, but we pick the first one for concreteness.)  
Let 
$$i = \min \{k = 1, 2, \dots, \tau-1: \mathtt{row}_k = \pi(\mathtt{row}_0), \text{ for some circular shift } \pi: \mathbb{Z}^\sigma \to \mathbb{Z}^\sigma\}$$
be the smallest $i$ such that $\mathtt{row}_i$ is a circular shift of $\mathtt{row}_0$, and let $i = 0$ if and only if $T$ does not have circular shifts of $\mathtt{row}_0$ other than this row itself.
Then this circular shift satisfies $\mathtt{row}_{(j + i) \mod \tau} = \pi(\mathtt{row}_j)$, for all $j = 0, \dots, \tau - 1$ and $i$ is determined by the tile $T$; 
we denote  this circular shift by $\pi_T^r$.
We denote by $\pi_T^c$ the analogous circular shift for columns.

\begin{lem} 
Let $T$ be a simple tile of a PS, and let $d_1 = \text{ord } (\pi_T^r)$ and $d_2 = \text{ord } (\pi_T^c)$. Then $d_1$ and $d_2$ are equal and divide  $\gcd(\tau, \sigma)$.
\end{lem}
\begin{proof}
Fix an element as $a_{0, 0}$.
By Lemma~\ref{lemma:simple-tiles-repeat}, parts 1 and 2,   $a_{0, 0}$ appears in $d_1$ rows of $T$.
It also appears in $d_2$ columns by Lemma~\ref{lemma:simple-tiles-repeat}, parts 3 and 4.
As a consequence, $d_1 = d_2$.
The divisibility follows from Lemma \ref{lemma: euler totient}.
\end{proof}

\begin{lem} \label{lemma: value of s(T)}
An integer $s\le n$ is the number of states in a simple tile $T = (a_{i,j})_{ i = 0,\dots, \tau-1, j = 0, \dots ,\sigma - 1}$ of PS if and only if there exists $d \bigm| \gcd (\tau, \sigma)$, such that $s = \tau\sigma/d$.  
\end{lem}
\begin{proof}
Let $T = (a_{i,j})_{ i = 0,\dots, \tau-1, j = 0, \dots , \sigma - 1}$.
Assume that $s(T) = s$ and let $d = \text{ord }(\pi_T^r)$.
Then by Lemma \ref{lemma:simple-tiles-repeat}, parts 1 and 2, the first $\tau/d$ rows of $T$ 
contain all states that are in $T$.
As a result, $s = \tau\sigma/d$ and $d = \text{ord }(\pi_T^r) \bigm| \gcd (\tau, \sigma)$.

Now assume that $d \bigm| \gcd (\tau, \sigma)$.
Then there exists a circular shift $\pi: \mathbb{Z}^\sigma \to \mathbb{Z}^\sigma$, such that $\text{ord }(\pi) = d$.
To form a simple tile $T$ with $s(T) = \tau\sigma/d$ states, construct a rectangle of $\tau/d$ rows and $\sigma$ columns using $\tau\sigma/d$ different states in the first $\tau/d$ rows of $T$.
Let $\mathtt{row}_{\tau/d}$ be defined by $\pi (\mathtt{row}_{0})$ and the subsequent rows are all automatically defined by the maps that are assigned in the first $\tau/d$ rows, by Lemma \ref{lemma: properties of tile}, part 1.
\end{proof}

The above lemma gives the possible values of $s(T)$ for a simple tile $T$ and the next one enumerates the number of simple tiles of PS containing $s$ different states.

\begin{lem} \label{lemma: counting simple tile}
The number of simple tiles of PS with temporal periods $\tau$ and spatial period $\sigma$ containing $s$ states is $\varphi(d){n \choose s}(s - 1)!$, where $d = \tau \sigma/s$.
\end{lem}
\begin{proof}
As in the proof of Lemma \ref{lemma: value of s(T)}, if $s(T) = s = \tau\sigma/d$, then $d = \text{ord} (\pi_T^r)$.
Moreover, there are $ {n \choose s} (s - 1)!$ ways to form the first $\tau/d$ rows of $T$.
Then, to uniquely determine $T$, we need to select a circular shift $\pi: \mathbb{Z}^\sigma \to \mathbb{Z}^\sigma$ with $\text{ord }(\pi) = d$ and define $\mathtt{row}_{\tau/d}$ to be $\pi (\mathtt{row}_{0})$.
By Lemma \ref{lemma: euler totient}, there are $\varphi(d)$ ways to do so.
\end{proof}

Consider two different simple tiles $T_1$ and $T_2$ under the rule.
As the final task of this section, we seek a lower bound on the combined number of values of the rule $f$ assigned by $T_1$ \textit{and} $T_2$, in terms of the number of states.
If $s(T_1) = s_1$, then $p(T_1) \ge s_1$, i.e., there are at least $s_1$ values assigned by $T_1$.
If there are $s_2^\prime$ states in $T_2$ that are not in $T_1$, then there are at least $s_2^\prime$ additional values to assign.
Therefore, a lower bound of the number of values to be assigned in $T_1$ and $T_2$ is $s_1 + s_2^\prime$.
The next lemma states that we can increase this lower bound by at least $1$ when $T_1 \cap T_2 \neq \emptyset$.
This fact plays an important role in the proofs of Theorem \ref{theorem: main 1} and Theorem \ref{theorem: main 2}.

\begin{lem}\label{lemma: one more map}
Let $T_1$ and $T_2$ be two different simple tiles 
for the same rule. If $T_1$ and $T_2$ have at least one state in common, then there exist $a_{i,j} \in T_1$ and $b_{k,m} \in T_2$ such that $a_{i,j} = b_{k,m}$ and $a_{i,j+1} \neq b_{k,m+1}$.
\end{lem}
\begin{proof}
As $T_1$ and $T_2$ have at least one state in common, we may pick $a_{i,j} \in T_1$ and $b_{k,m} \in T_2$, such that $a_{i, j} = b_{k, m}$.
If $a_{i, j + 1} \neq b_{k, m + 1}$, then we are done.
Otherwise, we repeat this procedure for $a_{i, j + 1}$ and $b_{k, m + 1}$ and see if $a_{i, j + 2} = b_{k, m + 2}$.
We repeat this procedure until we find two pairs such that $a_{i, j + q} = b_{k, m + q}$ and $a_{i, j + q + 1} \neq b_{k, m + q+1}$.
If we fail to do so, then $\mathtt{row}_i$ in $T_1$ and $\mathtt{row}_k$ in $T_2$ must be equal, up to a circular shift.
This implies that $T_1$ and $T_2$ must be the same since they are tiles for same rule, a contradiction.
\end{proof}

\section{Proofs of main results}

We will give a separate proof of Theorem~\ref{theorem: main 1}
first, for transparency, and then we show how to adapt the argument
to prove the stronger result,Theorem~\ref{theorem: main 2}.

\begin{proof}[Proof of Theorem~\ref{theorem: main 1}.]
We begin with the bounds
$$\mathbb{P}(\mathcal{P}_{\tau, \sigma, n}^{(0)} \neq \emptyset) 
\le \mathbb{P}(\mathcal{P}_{\tau, \sigma, n} \neq \emptyset) 
\le \mathbb{P}(\mathcal{P}_{\tau, \sigma, n}^{(0)} \neq \emptyset) + \sum_{\ell = 1}^{\tau\sigma} \mathbb{P}(\mathcal{P}_{\tau, \sigma, n}^{(\ell)} \neq \emptyset).$$
For $\ell\ge 1$, 
\begin{equation}\label{thm1pf-eq01}
\mathbb{P}(\mathcal{P}_{\tau, \sigma, n}^{(\ell)} \neq \emptyset) 
\le \mathbb{E(|\mathcal{P}}_{\tau, \sigma, n}^{(\ell)}|)
 =\sum_{s = 1}^{\tau\sigma}{n\choose s}g_{\tau, \sigma}^{(\ell)}(s)\frac{1}{n^{s + \ell}}
=\cO\left(\frac{1}{n^\ell}\right),
\end{equation}
where $g_{\tau, \sigma}^{(\ell)}(s)$ counts the number of $\tau \times \sigma$ tiles that contain $s$ different states and lag is $\ell$. 
Here, $1/n^{s+\ell}$ is the probability of such a tile (determined by a PS), as there are $s + \ell$ assignments to 
make by a random map, and each assignment occurs independently with probability $1/n$.
As $\ell \ge 1$, $\mathbb{P}(\mathcal{P}_{\tau, \sigma, n} \neq \emptyset) = \mathbb{P}(\mathcal{P}_{\tau, \sigma, n}^{(0)} \neq \emptyset) + \cO(1/n)$ as $n \to \infty$.

To find $\mathbb{P}(\mathcal{P}_{\tau, \sigma, n}^{(0)} \neq \emptyset)$ as $n \to \infty$, let $1= d_1<\ldots<d_u=\gcd(\sigma,\tau)$   be the common divisors of $\tau$ and $\sigma$ and $s_j = \tau\sigma/d_j$, for $j = 1, \dots, u$, be the possible numbers of states in simple tiles. 
We index the simple tiles that have $s_j$ states in an arbitrary way, so that $T_k^{(j)}$ be the $k$th simple tile that has 
$s_j$ states. Here $k = 1, \dots, N_j$ and $N_j=\varphi(d_j){n \choose s_j}(s_j - 1)!$ is the number of simple tiles with $s_j$ states (by Lemma \ref{lemma: counting simple tile}).
Let $I_k^{(j)}$ be the indicator random variable that $T_k^{(j)}$ is a tile determined by a PS. Let $W_n = \sum_{j=1}^u\sum_{k=1}^{N_j} I_k^{(j)}$ and 
\begin{align*}
\lambda_n 
& = \mathbb{E}W_n 
 = \sum_{j=1}^u\sum_{k=1}^{N_j} \mathbb{E}I_k^{(j)} 
 = 
\sum_{j = 1}^u\varphi(d_j){n\choose s_j} (s_j - 1)!\frac{1}{n^{s_j}}\\
& \xrightarrow{n\to\infty} \sum_{j = 1}^u\varphi(d_j)\frac{1}{s_j}\\
& = \sum_{j = 1}^u \varphi(d_j)\frac{d_j}{\tau \sigma}
 = \frac{1}{\tau\sigma}\sum_{d \bigm| \gcd(\tau, \sigma)} \varphi(d)d
 = \lambda_{\tau, \sigma}. 
\end{align*}
We next show that
$d_{\rm TV}(W_n, \Poi(\lambda_n)) \to 0$ as $n \to \infty$, which will conclude the proof. As orthogonal tiles have independent 
assignments, Lemma \ref{lemma: chen-stein lemma} implies that
\begin{equation}\label{thm1pf-eq1}
d_{TV}(W_n, \Poi(\lambda_n)) \le \min(1, \lambda_n^{-1})\left[\sum_{j, k}\left(\mathbb{E}I_k^{(j)}\right)^2 + 
\sum_{\substack{j,k, i, m \\ T_m^{(i)} \not\perp T_k^{(j)}}} \left( \mathbb{E}I_k^{(j)}\mathbb{E}I_m^{(i)} + \mathbb{E}I_k^{(j)}I_m^{(i)}\right) \right].
\end{equation}
To bound $\sum_{j, k}\left(\mathbb{E}I_k^{(j)}\right)^2$, fix a $j \in \{1, \dots, u\}$ and note that
\begin{equation}\label{thm1pf-eq11}
\sum_{k=1}^{N_j} \left(\mathbb{E}I_k^{(j)}\right)^2 = \varphi(d_j){n \choose s_j} (s_j - 1)! \left( \frac{1}{n^{s_j}}\right)^2 = \cO\left(\frac{1}{n^{s_j}}\right). 
\end{equation} 
It follows that $\sum_{j, k}\left(\mathbb{E}I_k^{(j)}\right)^2 = \cO\left(1/n^{\text{lcm}(\tau, \sigma)}\right) \to 0$, as $n \to \infty$.
It remains to bound the sum over $j,k,i,m$ in \ref{thm1pf-eq1}.
For a fixed $i, j\in \{1, \dots, u\}$, 
\begin{equation}\label{thm1pf-eq2}
\begin{aligned}
&\sum_{k = 1}^{N_j} \sum_{\substack{m = 1 \\ T_m^{(i)} \not\perp T_k^{(j)}}}^{N_i} 
\left(\mathbb{E}I_k^{(j)}\mathbb{E}I_m^{(i)}+
 \mathbb{E}I_k^{(j)}I_m^{(i)}\right)\\
& \le \sum_{k = 1}^{N_j} \sum_{\substack{m = 1 \\ T_m^{(i)} \cap T_k^{(j)} \neq \emptyset}}^{N_i} 
\left(\mathbb{E}I_k^{(j)}\mathbb{E}I_m^{(i)}+
 \mathbb{E}I_k^{(j)}I_m^{(i)}\right)\\
& = \sum_{k = 1}^{N_j} \sum_{h = 1}^{\min(s_i, s_j)} \sum_{\substack{m = 1 \\ |T_m^{(i)} \cap T_k^{(j)}|  = h}}^{N_i} \mathbb{E}I_k^{(j)}\mathbb{E}I_m^{(i)}
+ \sum_{k = 1}^{N_j} \sum_{h = 1}^{\min(s_i, s_j)} \sum_{\substack{m = 1 \\ |T_m^{(i)} \cap T_k^{(j)}|  = h}}^{N_i} \mathbb{E}I_k^{(j)}I_m^{(i)},\\
\end{aligned}
\end{equation}
where the inequality holds because two tiles that share an assignment have to share at least one state. Label the two 
triple sums on the last line of (\ref{thm1pf-eq2}) 
$S_{ij}^{(1)}$ and $S_{ij}^{(2)}$. 
Now, fix also an $h\in\{1, \dots, \min(s_i, s_j)\}$.
We first compute
\begin{align*}
\sum_{k = 1}^{N_j}  \sum_{\substack{m = 1 \\ |T_m^{(i)} \cap T_k^{(j)}| = h}}^{N_i} \mathbb{E}I_k^{(j)}\mathbb{E}I_m^{(i)} & = \varphi(d_j) {n \choose s_j}(s_j - 1)!\varphi(d_i){s_j \choose h}{n-s_j \choose s_i - h}(s_i - 1)!\frac{1}{n^{s_j}}\frac{1}{n^{s_i}}\\
& = \cO\left(\frac{1}{n^h}\right), 
\end{align*} 
and therefore $S_{ij}^{(1)}=\cO\left(1/n\right)$. Next, 
we estimate
\begin{align*}
\sum_{k = 1}^{N_j}  \sum_{\substack{m = 1 \\ |T_m^{(i)} \cap T_k^{(j)}|  = h}}^{N_i} \mathbb{E}I_k^{(j)}I_m^{(i)} & \le \varphi(d_j) {n \choose s_j}(s_j - 1)!\varphi(d_i){s_j \choose h}{n-s_j \choose s_i - h}(s_i - 1)!\frac{1}{n^{s_j}}\frac{1}{n^{s_i - h}} \frac{1}{n}\\
& = \cO\left(\frac{1}{n}\right), 
\end{align*}
and therefore $S_{ij}^{(2)}=\cO\left(1/n\right)$.
The inequality and the three powers of $n$ above are 
justified as follows: $1/n^{s_j}$ as there are $s_j$ states in in $T_k^{(j)}$, thus at least as many assignments; 
$1/n^{s_i - h}$ as there are $s_i - h$ states in $T_m^{(i)}$ that are not in $T_k^{(j)}$, thus at least as many assignments; 
and $1/n$ by Lemma \ref{lemma: one more map}, as $T_m^{(i)}$ and $T_k^{(j)}$ have $h \ge 1$ states in common and so there is at least one additional assignment. It follows that 
$d_{TV}(W_n, \Poi(\lambda_n))$ is bounded above by a constant 
times 
$$
\frac{1}{n^{\text{lcm}(\tau, \sigma)}}+\sum_{i,j}\left(S_{ij}^{(1)}+S_{ij}^{(2)}\right)= \cO\left(\frac{1}{n}\right), 
$$
which gives the desired result.
\end{proof}
 
We now give the proof of Theorem \ref{theorem: main 2},  
which mainly adds some notational complexity to the previous 
proof.

\begin{proof}[Proof of Theorem \ref{theorem: main 2}.]
Again, we begin with the bounds
$$\mathbb{P}(\mathcal{P}_{\Tau, \Sigma, n}^{(0)} \neq \emptyset) \le \mathbb{P}(\mathcal{P}_{\Tau, \Sigma, n} \neq \emptyset) \le \mathbb{P}(\mathcal{P}_{\Tau, \Sigma, n}^{(0)} \neq \emptyset) + \sum_{\ell \neq 0} \mathbb{P}(\mathcal{P}_{\Tau, \Sigma, n}^{(\ell)} \neq \emptyset),$$
where $\mathcal{P}_{\Tau, \Sigma, n}^{(\ell)}$ is the set of PS with periods $(\tau, \sigma) \in \Tau\times \Sigma$ whose tile has lag $\ell$.
Note that the summation is finite since $\Tau$ and $\Sigma$ are.
For $\ell\ge 1$, as in (\ref{thm1pf-eq01}), 
\begin{align*}
\mathbb{P}(\mathcal{P}_{\Tau, \Sigma, n}^{(\ell)} \neq \emptyset) 
&\le \sum_{(\tau, \sigma) \in \Tau\times \Sigma} \mathbb{P}(\mathcal{P}_{\tau, \sigma, n}^{(\ell)} \neq \emptyset)
 =\cO\left(\frac{1}{n^\ell}\right).
\end{align*}
As a consequence, $\mathbb{P}(\mathcal{P}_{\Tau, \Sigma, n} \neq \emptyset) = \mathbb{P}(\mathcal{P}_{\Tau, \Sigma, n}^{(0)} \neq \emptyset) + \cO(1/n)$ as $n \to \infty$.

To find $\mathbb{P}(\mathcal{P}_{\Tau, \Sigma, n}^{(0)} \neq \emptyset)$ as $n \to \infty$, we adopt the notation $u$, 
$d_j$, $s_j$, $T_k^{(j)}$ and $I_k^{(j)}$ from the proof of Theorem \ref{theorem: main 1}, for a fixed $\sigma$ and $\tau$. 
The dependence of these quantities  
on $\sigma$ and $\tau$ will be suppressed from the 
notation, as the periods are taken from a finite range and thus do not affect the computation that follows. Now,
$W_n = \sum_{(\tau, \sigma)}\sum_{j=1}^u\sum_{k=1}^{N_j} I_k^{(j)}$ and
$$\Lambda_n = \sum_{(\tau, \sigma)}\sum_{j=1}^u\sum_{k=1}^{N_j} \mathbb{E}I_k^{(j)} \to \sum_{(\tau, \sigma) \in \Tau \times \Sigma}\lambda_{\tau, \sigma} = \lambda_{\Tau, \Sigma},$$
as $n \to \infty$.
It remains to show that
$d_{TV}(W_n, \Poi(\Lambda_n)) \to 0$ as $n \to \infty$. From
Lemma~\ref{lemma: chen-stein lemma}, 
\begin{equation}\label{thm2pf-eq1}
\begin{aligned}
&d_{\rm TV}(W_n, \Poi(\Lambda_n))\\& \le \min(1, \Lambda_n^{-1})\left[\sum_{(\tau, \sigma)}\sum_{j, k}\left(\mathbb{E}I_k^{(j)}\right)^2 + \sum_{(\tau, \sigma)}\sum_{j, k} \sum_{(\tau', \sigma')} \sum_{\substack{i, m \\ T_m^{(i)} \not\perp T_k^{(j)}}} \left( \mathbb{E}I_k^{(j)}\mathbb{E}I_m^{(i)} + \mathbb{E}I_k^{(j)}I_m^{(i)}\right) \right].
\end{aligned}
\end{equation}
To bound
the double sum in (\ref{thm2pf-eq1}), observe that, for a 
fixed $(\tau,\sigma)$, the sum over $j,k$ is 
$\cO\left(1/n^{\lcm(\tau, \sigma)}\right)$ by 
(\ref{thm1pf-eq11}). As 
$\min_{\tau,\sigma}\lcm(\tau, \sigma)\ge 1$, the double sum 
in (\ref{thm2pf-eq1}) is $\cO(1/n)$.  
%
%

To bound the quadruple sum in (\ref{thm2pf-eq1}), fix a 
$(\tau, \sigma)$ for $I_k^{(j)}$, a $(\tau', \sigma')$ for $I_m^{(i)}$, and  $i, j\in \{1, \dots, u\}$. Then the  
sum over the remaining indices is bounded by $S_{ij}^{(1)}
+S_{ij}^{(2)}$, exactly as in (\ref{thm1pf-eq1}), except 
that now $S_{ij}^{(1)}$ and $S_{ij}^{(2)}$ also depend on the 
periods. The arguments that give $S_{ij}^{(1)}=\cO(1/n)$ and $S_{ij}^{(2)}=\cO(1/n)$ remain equally valid, and again 
imply $d_{TV}(W_n, \Poi(\Lambda_n))=\cO\left(1/n\right)$. 
\end{proof}

The proof of Corollary \ref{corollary: min period} is now straightforward.

\begin{proof}[Proof of Corollary \ref{corollary: min period}.]
Note that $\mathbb{P}(Y_{\sigma, n} \le y) = \mathbb{P}(\mathcal{P}_{[1, y], \{\sigma\}, n} \neq \emptyset) \to 1 - \exp\left(-\lambda_{[1, y], \{\sigma\}}\right)$, as $n \to \infty$, where $\lambda_{[1, y], \{\sigma\}} = \sum_{\tau = 1}^y \lambda_{\tau, \sigma}$. 
\end{proof}

For $\sigma = 1, 2, 3$ and $4$, the corresponding $\lambda_{\tau, \sigma}$ are 
$$
\lambda_{\tau, 1} =\frac{1}{\tau},\quad
\lambda_{\tau, 2} 
=
\begin{cases}
\frac{3}{2\tau} , & 2 \mid \tau\\
\frac{1}{2\tau} , & 2 \nmid  \tau
\end{cases},\quad
\lambda_{\tau, 3} 
= 
\begin{cases}
\frac{7}{3\tau} , & 3 \mid \tau\\
\frac{1}{3\tau} , & 3 \nmid  \tau
\end{cases}, \quad
\lambda_{\tau, 4} 
= 
\begin{cases}
\frac{11}{4\tau}, & \tau = 0 \mod 4\\
\frac{3}{4\tau}, & \tau = 2 \mod 4\\
\frac{1}{4\tau}, & \tau = 1, 3\mod 4 
\end{cases}.
$$
In Figure~\ref{figure: sims}, we present 
computer simulations to test how close the distribution of $Y_{\sigma, n}$ is to its limit for moderately large $n$
for the above four $\sigma$'s.
To compute $Y_{\sigma, n}(f)$, for every $f$ in the samples, 
we apply Algorithm \ref{algorithm: configuration digraph}.

\begin{figure}[ht!]%
    \centering
    \subfloat[{$\sigma=1$, $n=100$}]{{    \includegraphics[width = 0.45\textwidth]{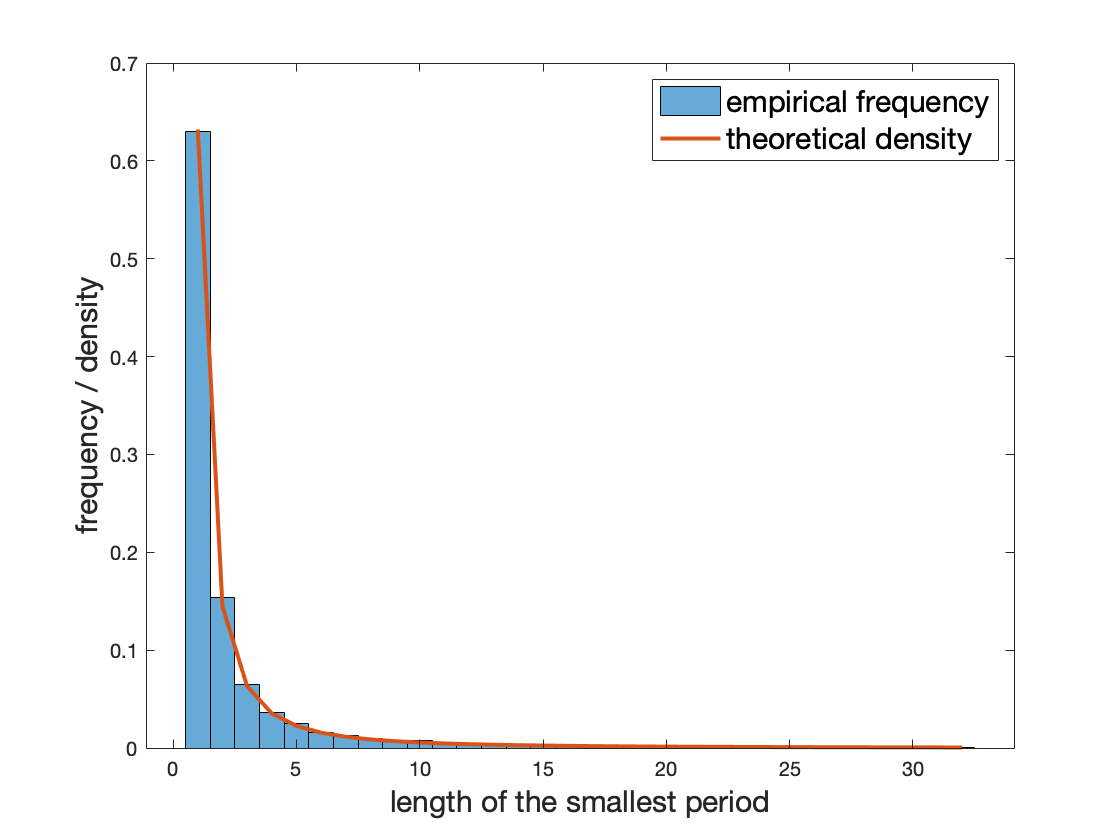} }}%
    \hspace{0cm}
    \subfloat[{$\sigma=2$, $n=100$}]{{ \includegraphics[width = 0.45\textwidth]{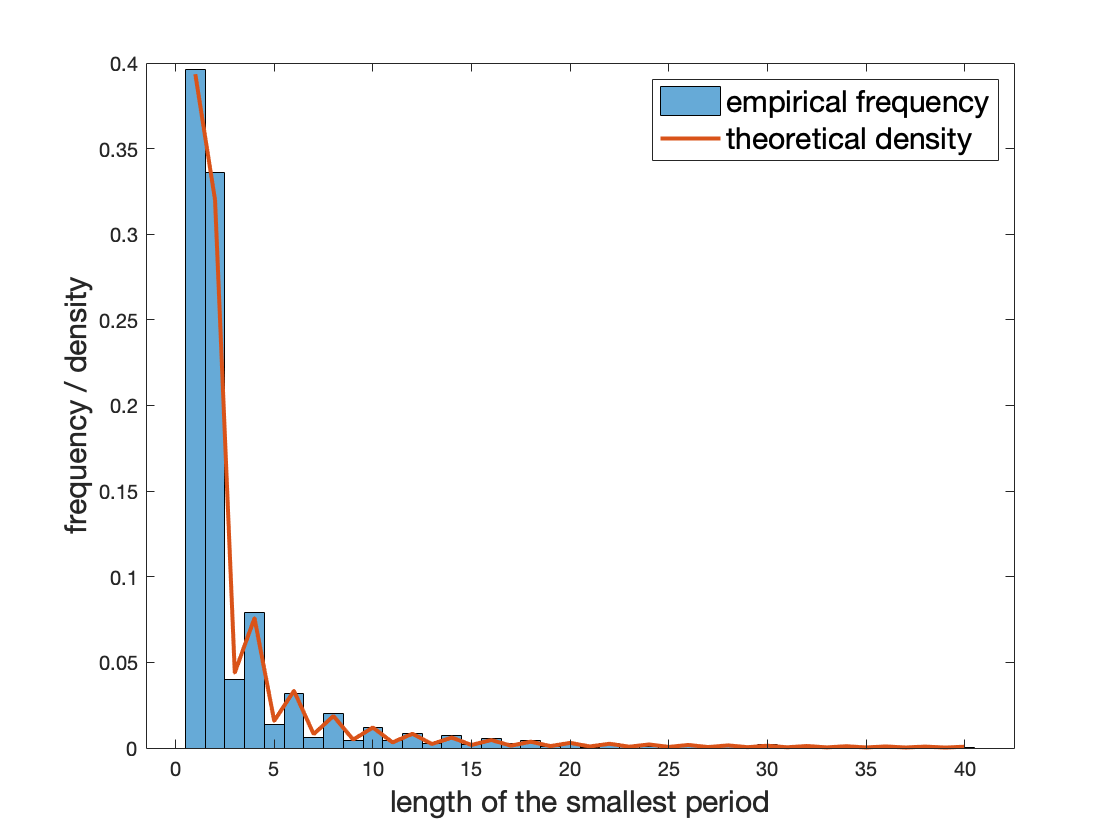} }}%
    \hspace{0cm}
    \subfloat[{$\sigma=3$, $n=60$}]{{\includegraphics[width = 0.45\textwidth]{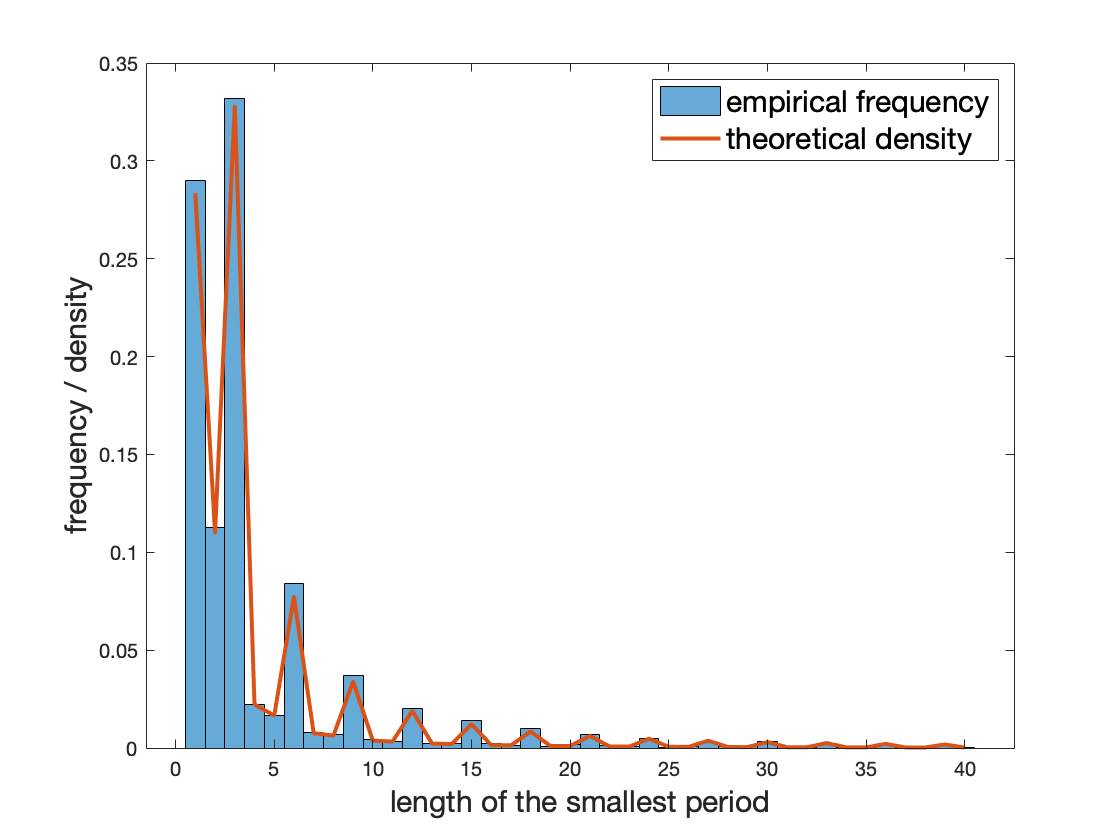} }}
    \hspace{0cm}
    \subfloat[{$\sigma=4$, $n=20$}]{{\includegraphics[width = 0.45\textwidth]{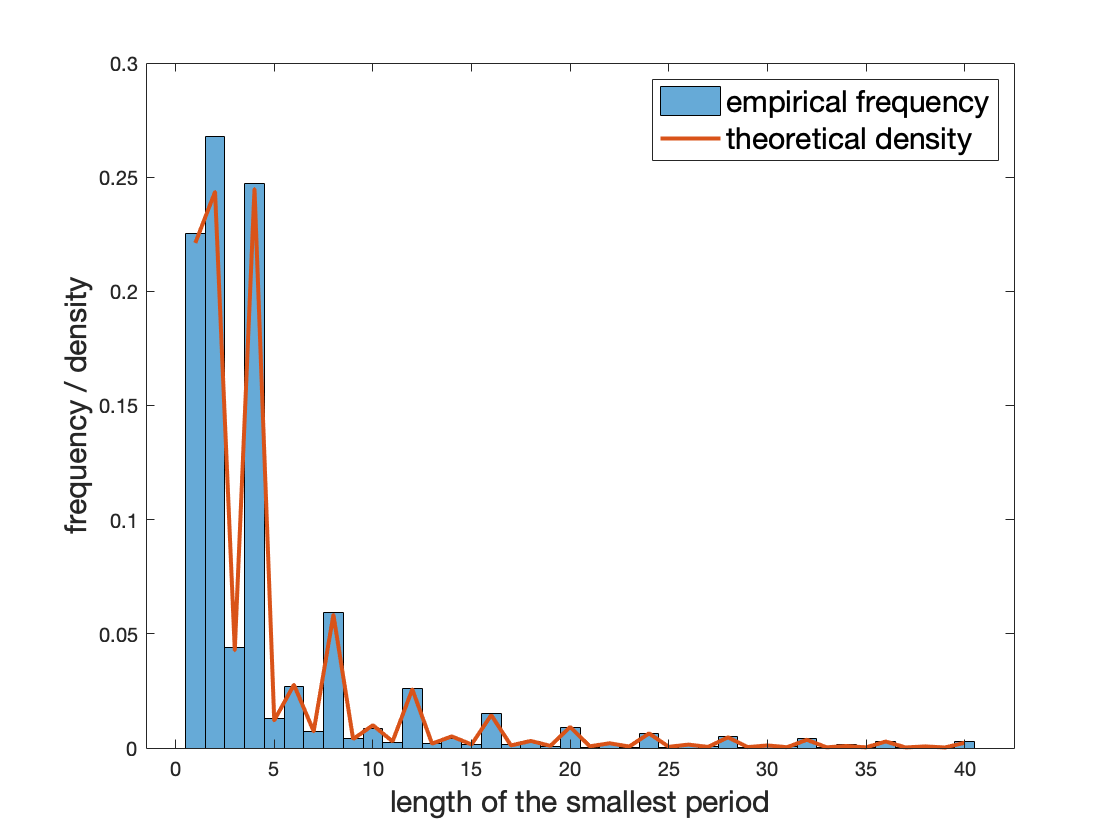} }}
    
    \caption{Lengths of the smallest temporal periods of PS with spatial periods $\sigma=1$ to $\sigma = 4$ and various $n$. In each case, a histogram from a random sample from 10,000 rules is compared to the theoretical limiting distribution as $n \to \infty$, given by Corollary~\ref{corollary: min period}.}%
    \label{figure: sims}%
\end{figure}

\section{Discussion and open problems}

In this paper, we initiate the study of periodic solutions for one-dimensional CA with random rules.  
Our main focus is the limiting probability of existence of a PS, when the rule is uniformly selected and the number of states approaches infinity, and we show  (Corollary \ref{corollary: min period}) that 
the smallest temporal period of PS with a given spatial period $\sigma$ is stochastically bounded.

By a similar argument, we can also obtain an analogous result in which we fix the temporal period instead of the spatial period. Define another random variable
$$
Y_{\tau, n}^\prime = \min\{\sigma: \mathcal{P}_{\tau, \sigma, n} \neq \emptyset\},
$$
which is the smallest spatial period of a PS given a temporal period $\tau$.
For example, for the four rules in Figure \ref{figure: 4 examples}, we may verify that, by Algorithm \ref{algorithm: PS}, 
$Y_{1, 3}^\prime (\mathit{012200210}) = 1$ ($0 \to 0$), 
$Y_{2, 3}^\prime (\mathit{021102120}) = 2$ ($12 \to 21 \to 12$), 
$Y_{3, 3}^\prime(\mathit{100112122}) = 3$ ($102\to 021 \to 210 \to 102$) and 
$Y_{4, 3}^\prime(\mathit{101201021}) = 4$
($0101 \to 2012 \to 1010 \to 0122 \to 0101$), with one cycle 
that generates the minimal PS given parenthetically for 
each case.  

\begin{cor}\label{cor:Y-prime}
The random variable $Y_{\tau, n}^\prime$ converges to a nontrivial distribution as $n \to \infty$.
\end{cor}

Perhaps the most natural generalization of Theorem \ref{theorem: main 2} would relax the condition that $\Tau$ and $\Sigma$ are finite.
The first case to consider surely is  when 
either $\Tau = \mathbb{N}$ or 
$\Sigma = \mathbb{N}$. For example,
it is clear that $\mathbb{P}(\mathcal{P}_{\mathbb{N}, \mathbb{N}, n} \neq \emptyset) = \mathbb{P}(\mathcal{P}_{\mathbb{N}, \{1\}, n} \neq \emptyset) = 1$, as any constant initial configuration 
eventually generates a PS with spatial period 1.

Now, consider a general $\sigma\ge 2$.
Let $\xi_0$ be a periodic configuration of spatial period $\sigma$.
Under any CA rule $f$, $\xi_1$ maintains the spatial periodicity, 
hence $\xi_t$ eventually enters into a PS, whose spatial period is however 
a divisor of $\sigma$, not necessarily $\sigma$ itself.
For this reason, we cannot reach an immediate conclusion 
about $\lim\mathbb{P}(\mathcal{P}_{\mathbb{N}, \{\sigma\}, n} \neq \emptyset)$, as $n \to \infty$.
We also refer the readers to \cite{gl2}, in which the reduction of temporal periods is explored in more detail.

For a fixed temporal period $\tau$, the matter is even less 
clear as a rule may not have a PS with temporal period that divides $\tau$.
For a trivial example with $\tau$ odd and $n = 2$, consider 
the ``toggle'' rule that always changes the current state and  
thus $\xi_{t+1}=1 - \xi_t$ and any initial state results 
in temporal period $2$. 
Thus we formulate the following intriguing open problem.

\begin{ques}
Let $\tau, \sigma \in \mathbb{N}$. 
What are the behaviors of $\mathbb{P}(\mathcal{P}_{\{\tau\}, \mathbb{N}, n} \neq \emptyset)$ and $\mathbb{P}(\mathcal{P}_{ \mathbb{N}, \{\sigma\}, n} \neq \emptyset)$, as $n \to \infty$ ?
\end{ques}

Another natural question addresses the case when $\sigma$ and $\tau$ increase with $n$.

\begin{ques}
For positive real numbers $a, b, c, d, \alpha, \beta, \gamma$ and $\delta$, what is the asymptotic behavior of $\mathbb{P}\left( \mathcal{P}_{I_1, I_2, n} \neq \emptyset\right)$, where $I_1 = [an^\alpha, bn^\beta]$ and $I_2 = [cn^\gamma, dn^\delta]$?
\end{ques}

A wider topic for further research is to 
investigate how different the behavior 
of the shortest temporal period changes if we choose a 
random rule from a subset of the set of all rules. There are, 
of course, many possibilities for such a subset, and we 
selected two natural ones below. In each case, we denote 
the resulting random variable with the same 
letter $Y_{n,\sigma}$.

A rule is \textbf{left permutative} if the map 
$\psi_b:\Z_n\to\Z_n$ given by $\psi_b(a)=f(a,b)$ is a permutation for every $b\in \Z_n$. 
Permutative rules, such as the 
famous {\it Rule 30\/}~\cite{wol-ran, jen3}, are 
good candidates for generation of long temporal periods. 

\begin{ques} Let $\mathcal L$ be the set of all $(n!)^n$ permutative rules. Choosing one of these rules uniformly at random from  $\mathcal L$, what is the asymptotic behavior of $Y_{n,\sigma}$?
\end{ques}

Our final question concerns the most widely studied special 
class of CA, the additive rules~\cite{martin1984algebraic}. Such  
a rule is given by $f(a,b)=\alpha a+\beta b$, for some 
$\alpha,\beta\in \Z_n$. 

\begin{ques} Let $\mathcal A$ be the set of all $n^2$ additive rules. Again, what is the asymptotic behavior of $Y_{n,\sigma}$
if a rule from $\mathcal A$ is chosen uniformly at random?
\end{ques}


\section*{Acknowledgements}
Both authors were partially supported by the NSF grant DMS-1513340.
JG was also supported in part by the Slovenian Research Agency (research program P1-0285). 

\bibliography{references}

\begin{thebibliography}{10}

\bibitem{barbour1992poisson}
Andrew~D. Barbour, Lars Holst, and Svante Janson.
\newblock {\em Poisson approximation}.
\newblock Clarendon Press, 1992.

\bibitem{boyle1999periodic}
Mike Boyle and Bruce Kitchens.
\newblock Periodic points for onto cellular automata.
\newblock {\em Indagationes Mathematicae}, 10(4):483--493, 1999.

\bibitem{boyle2007jointly}
Mike Boyle and Bryant Lee.
\newblock Jointly periodic points in cellular automata: computer explorations
  and conjectures.
\newblock {\em Experimental Mathematics}, 16(3):293--302, 2007.

\bibitem{cordovil1986periodic}
Raul Cordovil, Rui Dil{\~a}o, and Ana~Noronha da~Costa.
\newblock Periodic orbits for additive cellular automata.
\newblock {\em Discrete \& Computational Geometry}, 1(3):277--288, 1986.

\bibitem{gravner2012robust}
Janko Gravner and David Griffeath.
\newblock Robust periodic solutions and evolution from seeds in one-dimensional
  edge cellular automata.
\newblock {\em Theoretical Computer Science}, 466:64, 2012.

\bibitem{gl3}
Janko Gravner and Xiaochen Liu.
\newblock Maximal temporal period of a periodic solution generated by a
  one-dimensional cellular automaton.
\newblock {\em In preparation}, 2019.

\bibitem{gl2}
Janko Gravner and Xiaochen Liu.
\newblock One-dimensional cellular automata with random rules: longest temporal
  period of a periodic solution.
\newblock {\em In preparation}, 2019.

\bibitem{jen3}
Erica Jen.
\newblock Global properties of cellular automata.
\newblock {\em Journal of Statistical Physics}, 43(1--2):219--242, 1986.

\bibitem{jen1}
Erica Jen.
\newblock Cylindrical cellular automata.
\newblock {\em Communications in Mathematical Physics}, 118(4):569--590, 1988.

\bibitem{jen2}
Erica Jen.
\newblock Linear cellular automata and recurring sequences in finite fields.
\newblock {\em Communications in Mathematical Physics}, 119(1):13--28, 1988.

\bibitem{kim2009state}
Jae-Gyeom Kim.
\newblock On state transition diagrams of cellular automata.
\newblock {\em East Asian Math. J}, 25(4):517--525, 2009.

\bibitem{xiaochen}
Xiaochen Liu.
\newblock {\em Cellular automata with random rules}.
\newblock PhD thesis, University of California, Davis, in preparation, 2019.

\bibitem{martin1984algebraic}
Olivier Martin, Andrew~M. Odlyzko, and Stephen Wolfram.
\newblock Algebraic properties of cellular automata.
\newblock {\em Communications in mathematical physics}, 93(2):219--258, 1984.

\bibitem{misiurewicz2006iterations}
Micha{\l} Misiurewicz, John~G. Stevens, and Diana~M. Thomas.
\newblock Iterations of linear maps over finite fields.
\newblock {\em Linear algebra and its applications}, 413(1):218--234, 2006.

\bibitem{ross2011fundamentals}
Nathan Ross.
\newblock Fundamentals of stein's method.
\newblock {\em Probability Surveys}, 8:210--293, 2011.

\bibitem{stevens1999construction}
John~G. Stevens.
\newblock On the construction of state diagrams for cellular automata with
  additive rules.
\newblock {\em Information Sciences}, 115(1-4):43--59, 1999.

\bibitem{stevens1993transient}
John~G. Stevens, Ronald~E. Rosensweig, and A.~E. Cerkanowicz.
\newblock Transient and cyclic behavior of cellular automata with null boundary
  conditions.
\newblock {\em Journal of statistical physics}, 73(1-2):159--174, 1993.

\bibitem{wol-ran}
Stephen Wolfram.
\newblock Random sequence generation by cellular automata.
\newblock {\em Advances in Applied Mathematics}, 7(2):132--169, 1986.

\bibitem{wolfram2002new}
Stephen Wolfram.
\newblock {\em A new kind of science}, volume~5.
\newblock Wolfram media Champaign, IL, 2002.

\bibitem{xu2009dynamical}
Xu~Xu, Yi~Song, and Stephen~P. Banks.
\newblock On the dynamical behavior of cellular automata.
\newblock {\em International Journal of Bifurcation and Chaos},
  19(04):1147--1156, 2009.

\end{thebibliography}

\end{document}